\documentclass[
a4paper, 
oneside, 
headinclude,footinclude
]{amsart}

\usepackage[onehalfspacing]{setspace}
\usepackage[T1]{fontenc}
\usepackage[utf8]{inputenc}
\usepackage{lmodern}
\usepackage{cleveref}
\usepackage{stmaryrd}
\usepackage{graphicx}
\usepackage[english]{babel}
\usepackage{mathrsfs} 
\usepackage{amsfonts}
\usepackage{color}
\usepackage{amsthm}
\usepackage{amsmath}
\usepackage{enumerate}
\usepackage{tikz}
\usepackage{caption}
\usepackage{pgfplots}
\usepackage[all,cmtip]{xy}
\usepgfplotslibrary{polar}
\usepackage{verbatim}
\usepackage{amssymb}
\usepackage{dsfont}
\usepackage{mathtools}
\usepackage{lineno}
\theoremstyle{break}
\newtheorem{thm}{Theorem}[section]

\theoremstyle{definition}
\newtheorem{defi}[thm]{Definition}
\newtheorem{exm}[thm]{Example}
\theoremstyle{remark}

\newtheorem{rem}[thm]{Remark}
\newtheorem{prop}[thm]{Proposition}

\newcommand{\eps}{\varepsilon}

\newcommand{\orb}{\mathrm{orb}}
\newcommand{\C}{\mathrm{C}}

\newcommand{\phy}{\varphi}

\usepackage[
    backend=biber,
    natbib=true,
    firstinits=true,
    citestyle=ieee-alphabetic,
    bibstyle=alphabetic,
    url=false, 
    ,natbib		= true
    ,sorting		= nyt
    ,hyperref	= true
    ,useprefix	= true
 	,maxnames	= 5
   ,minnames	= 3
    ,firstinits	= true,
    doi=true,
    eprint=false
]{biblatex}
\renewbibmacro{in:}{}
\title{What can Koopmanism do for attractors in dynamical systems?}
\author{Viktoria Kühner}
\address{Viktoria Kühner, Mathematisches Institut, Universität Tübingen,
Auf der Morgenstelle 10, D-72076 Tübingen, Germany}
\email{viku@fa.uni-tuebingen.de}
\usepackage{fancyhdr}
\begin{document}
\begin{abstract}We characterize the longterm behavior of a semiflow on a compact space $K$ by asymptotic properties of the corresponding Koopman semigroup. In particular, we compare different concepts of attractors, such as asymptotically stable attractors, Milnor attractors and centers of attraction. Furthermore, we give a characterization for the minimal attractor for each mentioned property. The main aspect is that we only need techniques and results for linear operator semigroups, since the Koopman semigroup permits a global linearization for a possibly non-linear semiflow.\end{abstract}\vspace{1em}
\dedicatory{To Prof. S. H. Kulkarni on the occasion of his 65th birthday}
\setlength\parindent{0pt}
\maketitle

%%%%%%%%%%%%%%%%%%%%%%%%
%%%%%%%%%%%%%%%%%%%%%%%%
The concept of an \emph{attractor} of a dynamical system has been of great interest in the last 50 years. After the term first occurred in 1964 in \cite[p.55]{ausl}, cf. \cite[p.177]{milnor}, many variations and modifications have been defined, each of them yielding different examples. The survey article ``On the Concept of Attractor'' from 1985 by J. Milnor, cf. \cite{milnor}, treats the history of its many definitions and even adds an additional concept. He justifies this by describing some of the stability properties in the previous definitions as ``awkward'', \cite[p.178]{milnor}. In his opinion prior definitions seem to be too restrictive omitting some interesting examples, \cite[p.178]{milnor}. \\

The scope of this article is to establish a systematic hierarchy of the attractors mentioned in \cite{milnor} and, additionally, treat so-called \emph{minimal centers of attraction}. We will do so by ``translating'' these concepts into operator theoretic terms by linearizing the non-linear dynamical system globally. We call this process \emph{``Koopmanism''}. Indeed, this idea first appeared around 1930 in the papers \cite{koopman} by B. O. Koopman and \cite{neumann} by J. von Neumann and provided the precise mathematical terms to treat the so-called \emph{ergodic hypothesis} from L. Boltzmann formulated in \cite{boltzmann}. It is based on the distinction between the \emph{state space} $K$ of a (physical) system and the associated \emph{observable space} $\mathcal{O}$ being a (vector) space of real or complex valued functions on $K$. If the non-linear semiflow 
\[\varphi_t \colon K \to K \, , \ t\in \mathbb{R}_{+} \, , \] describes the dynamics on the state space, the maps 
\[f\mapsto T(t)f:=f\circ \varphi_t \, , \ f\in \mathcal{O} \, , \] become linear operators and, if $\mathcal{O}$ remains invariant, $(T(t))_{t\geq 0}$ is a one-parameter semigroup of linear operators on $\mathcal{O}$. \\

This idea led to the proof of the classical ergodic theorems of J. von Neumann \cite{neumannproof} and G. D. Birkhoff \cite{birkhoffproof} and even gave rise to \emph{ergodic theory} as a mathematical discipline.\\

The recent state of the art of this operator theoretic approach, called ``Koopmanism'', to ergodic theory is presented in the monograph \cite{erg}. In this paper we show how this can be used to classify and discuss attractors in topological dynamical systems. \\

For this purpose we consider a pair $(K,(\varphi_t)_{t\geq 0})$ consisting of a compact topological Hausdorff space $K$ and a continuous semiflow $(\varphi_t)_{t\geq 0}$ on $K$, cf. \cite[Def.3.1]{lecture}.
On the Banach space $\mathrm{C}(K)$ of all real-valued continuous functions on $K$ we obtain the corresponding $C_0$-semigroup $(T(t))_{t\geq 0}$, called \emph{Koopman semigroup}, given by the operators
\[T(t)f:=f\circ \varphi_t \textrm{ for } f\in \mathrm{C}(K)\, , \ t\geq 0 \, . \]

An \emph{attractor} is a closed and $(\varphi_t)_{t\geq 0}$-invariant subset $\emptyset\neq M \subseteq K$ possessing a certain asymptotic property. See \cite{milnor} for a survey of various concepts. In our perspective, every such subset of $K$ correlates with the closed and $(T(t))_{t\geq 0}$-invariant ideal $I_M:=\{f\in \mathrm{C}(K)\mid f|_M \equiv 0\}$ in the Banach algebra $\mathrm{C}(K)$. Essential to this matter is that all closed ideals in $\mathrm{C}(K)$ are of the form $I_M$ where $M$ is a closed subset of $K$ (cf. \cite[Theorem 4.8]{erg}). \\
{\begin{center}\begin{tabular}{lcr}$M\subseteq K$ closed subset & $\leftrightarrow$ & $I_M\subseteq \mathrm{C}(K)$ closed ideal \\ && \\
$M$ $(\varphi_t)_{t\geq 0}$-invariant & $\leftrightarrow$ & $I_M$ $(T(t))_{t\geq 0}$-invariant. \end{tabular}\end{center}}\vspace{1em}
Given an attractor $M$, our idea is to restrict the Koopman semigroup to the corresponding closed ideal $I_M$ and characterize the long-term behavior of $(\varphi_t)_{t\geq 0}$  around $M$ by asymptotic properties of $(T(t))_{t\geq 0}$ restricted to $I_M$. So our Leitmotiv can be explained as\\
{\begin{center}\begin{tabular}{lcr}``$\varphi_t\to M$'' & $\leftrightarrow$ & ``$T(t)\big|_{I_M}\to 0$'' \end{tabular}\end{center}}\vspace{1em}

The idea to characterize attractivity properties of invariant sets of a flow by stability properties of the Koopman operators restricted to functions vanishing on the attractor is due to A. Mauroy and I. Mezi\'{c}, see \cite[II.Prop.1]{mez}. Their stability corresponds to what we call weak stability later in this article.\\

In Section 1 we overview the stability theory for strongly continuous operator semigroups with more details on almost weak stability. In the following sections we then apply this theory to attractors in dynamical systems. In Section 2 we characterize absorbing sets and in Section 3 treat well-known attractivity and stability properties of dynamical systems by ``translating'' them into stability properties of the restricted Koopman semigroup. We then, in Section 4, prove the existence of minimal attractors and characterize these for each possible asymptotic behavior.\\

Now we recall some basic facts and fix the notation. Let $K$ be a compact Hausdorff space. A family of self-mappings $(\varphi_t)_{t\geq 0}$ on $K$ 
is called \emph{semiflow} if $\varphi_0\equiv id_K$ and $\varphi_{t+s}=\varphi_t\circ \varphi_s$ for all $s, t \geq 0$. We call $(\varphi_t)_{t\geq 0}$ a \emph{continuous semiflow} if 
\begin{align*} [0,\infty)\times K &\to K \, , \\
(t,x)&\mapsto \varphi_t(x) 
\end{align*}  is continuous with respect to the product topology. Thus, a \emph{dynamical system} is a pair $(K,(\varphi_t)_{t\geq 0})$ consisting of a compact Hausdorff space $K$ and a continuous semiflow $(\varphi_t)_{t\geq 0}$. We also call $K$ the \emph{state space} and the elements $x\in K$ \emph{states}. The induced Koopman semigroup $(T(t))_{t\geq 0}$ on $(\mathrm{C}(K),\|\cdot\|_\infty)$, as defined above, is strongly continuous if and only if $(\varphi_t)_{t\geq 0}$ is continuous, cf. \cite[B-II, Lem. 3.2]{lecture}.\\

We recall that $(\mathrm{C}(K),\|\cdot\|_\infty)$ is a C$^*$-algebra and a Banach lattice for the usual pointwise operations. Given a function $f\colon K\to \mathbb{R}$ and $a\in \mathbb{R}$ we use the notation 
\[ [f<a]:= f^{-1}((-\infty,a))\, , \ [f\leq a]:=f^{-1}((-\infty,a])\, , \ [f=a]:=f^{-1}(\{a\}) \] and, analoguously, $[f>a]$ and $[f\geq a]$.
The sets $[|f|> 0]$, $f\in\mathrm{C}(K)$, form a basis for the topology on $K$ since $K$ is completely regular, \cite[Appendix A.2]{erg} and \cite[Proof of Lem. 4.12]{erg}. This is equivalent to the fact that the zero sets $[f=0]$, $f\in \mathrm{C}(K)$, form a basis of the closed subsets of $K$ or that the topology on $K$ coincides with the initial topology induced by $\mathrm{C}(K)$. Combining these facts, given a closed subset $M\subseteq K$ and $U$ an open neighborhood of $M$ there exists $f\in\mathrm{C}(K)$ with $M\subseteq [f=0]$ and $\varepsilon >0$ such that $[|f|<\varepsilon]\subseteq U$, i.e. the sets of the form $U_{\varepsilon,f}:=[|f|<\varepsilon]$, $f\in \mathrm{C}(K)$, $f(M)=0$ and $\varepsilon>0$ form a basis for the system of neighborhoods of $M$ which we will denote by $\mathcal{U}(M)$.\\

As already mentioned above, for every closed ideal $I\subseteq \mathrm{C}(K)$ there exists a closed subset $M\subseteq K$ such that
\[I=I_M=\{f\in \mathrm{C}(K) \mid f|_M\equiv 0 \} \textrm{ and }M=\bigcap\limits_{f\in I}[f=0] \, , \]  see \cite[Thm. 4.8]{erg}. The closed set $M$ is said to be $(\varphi_t)_{t\geq 0}$-invariant if $\varphi_t(M)\subseteq M$ for all $t\geq 0$. It is $(\varphi_t)_{t\geq 0}$ invariant if and only if the corresponding ideal $I_M$ is $(T(t))_{t\geq 0}$-invariant, \cite[Lem. 4.18]{erg}. We remark that for $M\subseteq K$ closed, $I_M\cong \mathrm{C}_0(K\setminus M)$ by $f\mapsto f|_{K\setminus M}$, where $\mathrm{C}_0(K\setminus M)$ is the space of all real-valued continuous functions on $K\setminus M$ that vanish at infinity. We identify the dual spaces $\mathrm{C}(K)'$ and $I_M' \cong \mathrm{C}_0(K\setminus M)'$ of $\mathrm{C}(K)$ and $I_M$ with the regular Borel measures on $K$ and $K\setminus M$ respectively. For a subset $M\subseteq K$ we either write $M^c$ or $K\setminus M$ for its complement in $K$.\\

%The theory of closed invariant subsets of $K$ covers naturally the theory of local attractors in dynamical systems on \emph{locally-compact spaces}. 
%Considering a dynamical system on a locally-compact space $\Omega$, an \emph{attractor} is usually defined as a compact invariant subset $\emptyset \neq M \subset \Omega$ such that the \emph{region (or basin) of attraction}, i.e., the subset of $\Omega$ on which ``$\varphi_t\to M$'' in a certain way, is an invariant neighborhood of $M$, \cite[Def. 1.5, Chapt. V]{bhatia}, \cite[]{milnor}. Since $M$ is compact we can without loss of generality assume the region of attraction to be compact. Lets denote the basin of attraction by $K$. We obtain 
%\[\mathrm{C}_0(\Omega)/ I_K \cong \mathrm{C}(K) \] where $\mathrm{C}_0(\Omega)$ is the space of all continuous functions on $\Omega$ that vanish at infinity and can thus interpret $I_M$ as a closed ideal of $\mathrm{C}(K)$. Thus, it suffices to consider the Koopman semigroup on $\mathrm{C}(K)$ to study local attractors in dynamical systems on locally-compact spaces. For global attractors in $\Omega$ the results for $\mathrm{C}(K)$ can easily be extended to this case by defining the corresponding Koopman semigroup on $\mathrm{C}_0(\Omega)$.
%%%%%%%%%%%%%%%%%%%%
%%%%%%%%%%%%%%%%%%%%
\section{Stability of $C_0$-semigroups}
In this section we recall various stability properties of $C_0$-semigroups on a Banach space $X$. The concept of \emph{almost weak stability} is treated in more detail.
\begin{defi}\label{stability}Let $(T(t))_{t\geq0}$ be a strongly continuous semigroup on a Banach space $X$. Then $(T(t))_{t\geq 0}$ is said to be
\begin{itemize}
\item[a)] \emph{nilpotent} if there exists $t_0>0$ such that
\[\|T(t_0)\|=0\, ,\]
\item[b)] \emph{uniformly exponentially stable} if there exists $\delta>0$ such that 
\[\lim\limits_{t\to\infty}e^{-\delta}\|T(t)\|=0 \, ,\]
\item[c)] \emph{uniformly stable} if
\[\lim\limits_{t\to\infty}\|T(t)\|=0 \, , \]
\item[d)] \emph{strongly stable} if
\[\lim\limits_{t\to\infty}\|T(t)x\|=0 \textrm{ for all }x\in X \, ,\]
\item[e)]  \emph{weakly stable} if
\[\lim\limits_{t\to\infty}\langle T(t)x,x'\rangle=0  \textrm{ for all }x\in X\, , \ x' \in X' \textrm{ and}  \]
\item[f)] \emph{almost weakly stable} if for all pairs $(x,x') \in X\times X'$ there exists a subset $R\subset \mathbb{R}_+$ with density \footnote{The density of a subset $R\subset \mathbb{R}_+$ is 
\[d(R):=\lim\limits_{t\to \infty} \frac{1}{t} \lambda\left([0,t]\cap R\right),\ \lambda \textrm{ Lebesgue measure, }  \] if the limit exists.} $1$ such that
\[\lim\limits_{t\to\infty, t\in R}\langle T(t)x,x'\rangle=0 \, .  \]
\end{itemize}

\end{defi}
In the above definition the following chain of implications holds\[ {\normalfont{a)\implies b)\implies c)\implies d)\implies e) \implies f)}}\, . \] All implications are strict except {\normalfont{b)$\iff$c)}} which can be found in \cite[Chapter V, Section 1]{engnag}. For examples we refer to \cite[Chapter III]{tanja} and \cite[Chapter V, Section 1]{engnag}. \\[1em]
To later study stability of Koopman semigroups on $\mathrm{C}(K)$-spaces we need an additional definition.
\begin{defi}Let $(K,(\varphi_t)_{t\geq 0})$ be a dynamical system, $(T(t))_{t\geq 0}$ the corresponding Koopman semigroup on $\mathrm{C}(K)$. Then $(T(t))_{t\geq 0}$ is said to be \emph{almost everywhere pointwise stable} if there exist a quasi invariant\footnote{A measure $\mu$ on $K$ is called \emph{quasi invariant with respect to $(\varphi_t)_{t\geq 0}$} if $\mu(\varphi_t^{-1}(N))=0$ for all $t\geq 0$ if and only if $\mu(N)=0$.} regular Borel measure on $K$ such that for every $f\in \mathrm{C}(K)$
\[T(t)f(x)\to 0 \textrm{ as }t\to \infty \textrm{ for $\mu$-almost all }x\in K\, . \]
\end{defi}
We also recall the definition of the \emph{growth bound} $\omega_0$ of a strongly continuous semigroup $(T(t))_{t\geq 0}$ on a Banach space $X$ which is defined as 
\[\omega_0:=\inf\{\omega \in \mathbb{R}\mid \|T(t)\|\leq Me^{\omega t}\textrm{ for all }t\geq 0\textrm{ and suitable }M>0 \} \, . \]

\subsection{Almost Weak Stability}
For a complete treatment of almost weak stability for $C_0$-semigroups on Banach spaces with relatively weakly compact orbits we refer to \cite[Chapter III, Section 5]{tanja} and, for a time-discrete variant to \cite[Chapter 9]{erg}. The tools and ideas used in this subsection are based on \cite[Chapter 9]{erg}.
\begin{prop}\label{fss}Let $(T(t))_{t\geq 0}$ be $C_0$-semigroup of contractions on some Banach space $X$. Then the following are equivalent
\begin{itemize}
\item[a)] $(T(t))_{t\geq 0}$ is almost weakly stable,
\item[b)] \[\lim\limits_{T\to \infty} \frac{1}{T}\int\limits_0^T \! |\langle T(t)x,x'\rangle|\, \mathrm{d}t=0 \] for all $x\in X$, $x'\in X'$,
\item[c)] \[\lim\limits_{T\to \infty} \sup\limits_{x'\in X', \|x'\|\leq 1}\frac{1}{T}\int\limits_0^T \! |\langle T(t)x,x'\rangle|\, \mathrm{d}t=0 \] for all $x\in X$.
\end{itemize}
\end{prop}
\begin{proof}
The equivalence {\normalfont{a)$\iff$b)}} follows from the so called \emph{Koopman-von Neumann Lemma}, see for example \cite[Chapter III, Lemma 5.2]{tanja}. The implication {\normalfont{b)$\implies$c)}} in the time discrete analogue is due to \emph{Jones and Lin}, cf. \cite{joneslin}. We adapt the proof given in \cite[Prop. 9.17]{erg}. Every operator $T(t)$ as its adjoint $T(t)'$ is a contraction and the dual unit ball $B'$ is compact with respect to the weak-* topology. Due to these facts we can define the Koopman system 
\[(\C(B'),(\tilde{T}(t))_{t\geq 0})\, \] where 
\[\tilde{T}(t)f(x'):=f(T(t)'x') \, \] for $t\geq 0$, $f\in \C(B')$, $x'\in B'$. Fix $x\in X$ and define $g_x\in \C(B')$ by $g_x(x'):=|\langle x, x' \rangle|$. By {\normalfont{b)}}  \[\lim\limits_{T\to\infty} \frac{1}{T}\int\limits_0^T \! \tilde{T}(t)g_x(x')\, \mathrm{d}t=0 \] pointwise in $x'$ and by Lebesgue's theorem of dominated convergence also weakly and thus in the norm of $\C(B')$, cf. \cite[Chapter I, Theorem 2.25]{tanja} or \cite[Proposition 8.18]{erg}.
\end{proof}
\begin{rem}\label{awsi}Let $(T(t))_{t\geq 0}$ be $C_0$-semigroup of contractions on some Banach space $X$. The subset 
\[I_{\textrm{aws}}:=\{x \in X \mid \lim\limits_{T\to \infty} \sup\limits_{x'\in X', \|x'\|\leq 1}\frac{1}{T}\int\limits_0^T \! |\langle T(t)x,x'\rangle|\, \mathrm{d}t=0\} \]
is a closed, $(T(t))_{t\geq 0}$-invariant subspace of $X$.
\end{rem}
\begin{proof} Let $(x_n)_{n\in\mathbb{N}}$ be a convergent sequence in $I_{\textrm{aws}}$ with limit $x\in X$ and take $\varepsilon >0$. Then there exists $n_0\in \mathbb{N}$ such that $\|x_n-x\|< \frac{\varepsilon}{2}$ for all $n\geq n_0$. Now consider $j\geq n_0$. By \Cref{fss} {\normalfont{c)}} there exists $t(j)\geq0$ such that 
\[ \sup\limits_{x'\in X', \|x'\|\leq 1}\frac{1}{T}\int\limits_0^T \! |\langle T(t)x_j,x'\rangle|\, \mathrm{d}t < \frac{\varepsilon}{2} \] for all $T>t(j)$. This implies
\begin{align*}&\sup\limits_{x'\in X', \|x'\|\leq 1}\frac{1}{T}\int\limits_0^T \! |\langle T(t)x,x'\rangle|\, \mathrm{d}t \\ = &\sup\limits_{x'\in X', \|x'\|\leq 1}\frac{1}{T}\int\limits_0^T \! |\langle T(t)(x_j-x),x'\rangle|\, \mathrm{d}t + \sup\limits_{x'\in X', \|x'\|\leq 1}\frac{1}{T}\int\limits_0^T \! |\langle T(t)x_j,x'\rangle|\, \mathrm{d}t\\
\leq &\|x-x_j\| +\sup\limits_{x'\in X', \|x'\|\leq 1}\frac{1}{T}\int\limits_0^T \! |\langle T(t)x_j,x'\rangle|\, \mathrm{d}t <\varepsilon \textrm{ for all }T\geq \max\{n_0,t(j)\} \, . 
\end{align*}
\end{proof}
\begin{rem}Let $(T(t))_{t\geq 0}$ be $C_0$-semigroup of contractions on some Banach space $X$. In analogy to the previous remark, we define \begin{align*}I_{\textrm{ss}}&:=\{x\in X\mid \|T(t)x\|\to 0 \textrm{ as }t\to \infty \} \textrm{ and } \\
I_{\textrm{ws}}&:=\{x\in X \mid \langle T(t)x, x'\rangle \to 0 \textrm{ as }t\to \infty \textrm{ for all }x'\in X' \} \, . \end{align*} Both are clearly closed subspaces of $X$.
\end{rem}
\begin{rem}Let $(K,(\varphi_t)_{t\geq 0})$ be a dynamical system, $(T(t))_{t\geq 0}$ the corresponding Koopman semigroup on $X:=\mathrm{C}(K)$ and $\mu$ a quasi invariant regular Borel measure on $K$. We define
\[I_{\textrm{aews}}:=\{f\in X \mid T(t)f(x)\to 0\textrm{ as }t\to \infty \ \textrm{ for $\mu$-almost all }x\in K\}\, . \] Which is a closed subspace of $X$. 
\end{rem}
\begin{prop}\label{ja} Let $(T(t))_{t\geq 0}$ be a $C_0$-semigroup of contractions on a Banach space $X$. If for all $x\in X$ there exists a sequence $(t_n)_{n\in \mathbb{N}}$ in $[0,\infty)$ with $t_n\to \infty$ as $n \to \infty$ such that for all $x'\in X'$ 
\[\lim\limits_{n\to\infty} \langle T(t_n)x,x'\rangle =0 \, ,  \] then $(T(t))_{t\geq 0}$ is almost weakly stable.
\end{prop}
\begin{proof}Take $x\in X$ and $(t_n)_{n\in \mathbb{N}}$, $t_n\to\infty$ such that 
\[\lim\limits_{n\to\infty} \langle T(t_n)x,x'\rangle =0 \,  \] for all $x'\in X'$. As in the proof of \Cref{fss} we consider the induced Koopman system $(\C(B'),(\tilde{T}(t))_{t\geq 0})$ and the function 
\[g_x(x'):=|\langle x,x'\rangle| \, . \] 
If $\mu\in \C(B')'$ vanishes on $\bigcup\limits_{t\geq 0}(\mathrm{rg}(\mathrm{Id}-\tilde{T}(t)))$, then 
\begin{align*}\langle g_x,\mu \rangle = \langle \tilde{T}(t_n)g_x, \mu\rangle
\end{align*} for all $n\in \mathbb{N}$. We observe that
\begin{align*}\langle \tilde{T}(t_n)g_x,\mu \rangle &= \int_{B'}\! \tilde{T}(t_n)g_x(x')\, \mathrm{d}\mu(x')\\
&= \int_{B'}\! |\langle x, T(t_n)'x'\rangle | \, \mathrm{d}\mu(x')\\
&= \int_{B'}\! |\langle T(t_n) x, x'\rangle | \, \mathrm{d}\mu(x')\, .
\end{align*}
The functions $|\langle T(t_n) x, x'\rangle | $ converge to $0$ pointwise in $x'$ by assumption and by Lebesgue's Theorem the integral $\int_{B'}\! |\langle T(t_n) x, 'x'\rangle | \, \mathrm{d}\mu(x')$ goes to $0$ as well. This implies $g_x\in \overline{\mathrm{lin}\bigcup\limits_{t\geq 0}(\mathrm{rg}(\mathrm{Id}-\tilde{T}(t)))}$ by the theorem of Hahn-Banach. Thus,
\begin{align*}&\frac{1}{T}\int_0^T\! |\langle T(t)x,x'\rangle| \, \mathrm{d}t \\
=&\frac{1}{T}\int_0^T\! (\tilde{T}(t)g_x)(x') \, \mathrm{d}t \xrightarrow[]{T\to\infty} 0
\end{align*} for all $x'\in X'$. Since $x$ was arbitrary, $(T(t))_{t\geq 0}$ is almost weakly stable.
\end{proof}
\begin{rem}For $C_0$-semigroups of contractions with relatively weakly compact orbits the assertions in \Cref{ja} are equivalent, see \cite[Chapter III, Section 5]{tanja}. Here we are able to prove one implication without assuming relatively weakly compact orbits. It still remains open if the opposite implication also holds true in this case.
\end{rem}
%%%%%%%%%%%%%%%%%%%%%%%%
From now on $(K,(\varphi_t)_{t\geq 0})$ is a topological dynamical system with compact state space $K$ if not otherwise stated and $(T(t))_{t\geq 0}$ denotes the induced Koopman semigroup on $\mathrm{C}(K)$.
\section{Absorbing sets}
The following section is dedicated to \emph{absorbing sets}, these are compact invariant subsets of the state space that eventually contain every initial state. We follow the definition in \cite[Def. 2.1.1]{chueshov} and differ between two types of such sets.
\begin{defi}A closed invariant set $ M \subsetneq K$ is called
\begin{itemize}
\item[a)]\emph{absorbing} if there exists $t_0>0$ such that
\[\varphi_{t_0}(K)\subseteq M \, , \]
\item[b)]\emph{pointwise absorbing} if for all $x\in K$ there exists $t_0>0$ such that
\[\varphi_{t_0}(x)\in M \, . \]
\end{itemize}
\end{defi}
This gives rise to the notion of \emph{dissipative systems}, cf. \cite[Def. 2.1.1]{chueshov}.
\begin{defi} A dynamical system $(K, (\varphi_t)_{t\geq 0})$ is called \emph{(point) dissipative} if it contains a (point) absorbing set. 
\end{defi}

\begin{prop}Let $M \subsetneq K$ be a closed invariant set and $(S(t))_{t\geq 0}$ the restricted Koopman semigroup, i.e. $S(t):=T(t)|_{I_M}$ for $t\geq 0$. Then all the assertions in {\normalfont{(I)}} and all the assertions in {\normalfont(II)} are equivalent.
\begin{itemize}
\item[(I)]
\begin{itemize}
\item[a)] $(S(t))_{t\geq 0}$ is nilpotent.
\item[b)] $(S(t))_{t\geq 0}$ is uniformly stable.
\item[c)] $\omega_0=-\infty$.
\item[d)] $M$ is absorbing.
\end{itemize}
\item[(II)]
\begin{itemize}
\item[a)] For all Dirac measures $\delta_x\in C(K)'$ there exists $t_0>0$ such that 
\[S(t_0)'\delta_x=0 \, . \]
\item[b)] $M$ is pointwise absorbing.
\end{itemize}
\end{itemize}
\end{prop}
\begin{proof} We begin with the proof of {\normalfont{(I)}}. Clearly, {\normalfont{a)$\implies$b)}}. For the implication {\normalfont{b)$\implies $d)}} assume $M$ not to be absorbing and fix $t_0>0$, thus there is $x_0\in K\setminus M$ with $\varphi_{t_0}(x_0)\in K\setminus M$. Since $K\setminus M$ is completely regular there exists $f\in \C_0(K\setminus M)\cong I_M$ with $\|f\|=1$ and $f(\varphi_{t_0}(x_0))=1$. Therefore, 
\[\|S(t_0)\|\geq \|S(t_0)f\| \geq S(t_0)f(x_0)=1 \, .\]
Since $t_0$ was arbitrary $\|S(t)\|= 1$ for all $t\geq 0$ which contradicts {\normalfont{b)}}. 
The implication {\normalfont{d)$\implies$a)}} can be seen as follows. Let $t_0>0$ be such that $\varphi_{t_0}(K)\subseteq M$, thus $S(t_0)f(x)=f(\varphi_{t_0}(x))=0$ for every $f\in I_M$. This implies $\|S(t_0)\|=\sup\limits_{\|f\|\leq 1}\|S(t_0)f\|=0$.
Additionally, clearly {\normalfont{a)}} implies {\normalfont{c)}}.\\[1em]
Proof of {\normalfont{(II)}}: These equivalences are quite clear since {\normalfont{a)}} implies that for all $x\in K$ there exists $t_0>0$ such that 
\[\varphi_{t_0}(x)\in \bigcap\limits_{f\ in I_M}[f=0]=M \, . \]

\end{proof}
Next we give a condition under which the two concepts coincide. We recall that a compact space $K$ is a Baire space, thus for a sequence of closed subsets $K_n$, $n\in\mathbb{N}$, with 
\[K =\bigcup\limits_{n\in\mathbb{N}} K_n\] there exists $n\in \mathbb{N}$ such that $K_n$ has non-empty interior. \\

\begin{rem}\label{Kn}Let $ M \subsetneq K$ be closed and invariant. If $M$ is pointwise absorbing the sets $\varphi_n^{-1}(M)=:K_n$ form a closed cover of $K$.
\end{rem}

\begin{prop}Let $ M \subsetneq K$ be closed and invariant and $K_n$ as defined in \Cref{Kn}. The set $M$ is absorbing if and only if it is pointwise absorbing and $M\subset \overset{\circ}{K_n}$ for some $n\in \mathbb{N}$.
\end{prop}
\begin{proof}Clearly, if $M$ is absorbing it is pointwise absorbing and there exists $n\in\mathbb{N}$ such that $M\subseteq K =\overset{\circ}{K_n}$ in above construction. For the other implication consider the following. By assumption for every $x\in K$ there exists $t_x\geq 0$ such that 
\[\varphi_{t_x}(x)\in \overset{\circ}{K_n}\subset K_n \, . \] By continuity $\varphi_{t_x}^{-1}(\overset{\circ}{K_n})$ is open for every $x\in K$ and 
\[K\subseteq \bigcup\limits_{x\in K} \varphi_{t_x}^{-1}(\overset{\circ}{K_n})\, . \] Since $K$ is compact there exist finitely many $x_1,\dots,x_j$ for some $j\in \mathbb{N}$ such that 
\[K\subseteq \bigcup\limits_{k=1}^j \varphi_{t_{x_k}}^{-1}(\overset{\circ}{K_n})\, . \] This implies for $y\in K$ that \[\varphi_{t_{x_k}}(y)\in \overset{\circ}{K_n}\subset K_n \] for some $k\in\{1,\dots,j\}$ and therefore 
\[\varphi_{t_{x_k}+n}(y)\in M \, . \] Define $T:=\max \{t_{x_k}\mid k\in\{1,\dots,j\}\}$, then 
\[\varphi_{T+n}(y)\in M \] by invariance of $M$. 
\end{proof}
\section{Asymptotics of Dynamical Systems}
In this section we consider asymptotic properties of semiflows around closed invariant sets and give operator theoretic characterization of each such property. We also discuss correlations between them. 
\begin{defi}\label{DS} A closed invariant set $\emptyset\neq M \subseteq K$ is called
\begin{itemize}
\item[a)] \emph{uniformly attractive} if for all $U\in\mathcal{U}(M)$ there exists $t_0>0$ such that \[\varphi_t(K)\subseteq U \textrm{ for all }t\geq t_0 \, , \]
\item[b)] \emph{(pointwise) attractive} if for all $x\in K$ and $U\in\mathcal{U}(M)$ there exists $t_0>0$ such that
\[\varphi_t(x)\in U \textrm{ for all } t\geq t_0 \, , \]
\item[c)] \emph{likely limit set (or Milnor attractor)} if there exists a quasi invariant Borel measure $\mu$ on $K$ such that for all $U\in \mathcal{U}(M)$ and $\mu$-almost every $x\in K$ there exists $t_0>0$ with
\[\varphi_t(x)\in U \textrm{ for all } t\geq t_0 \, , \]
\item[d)]\emph{center of attraction} if for all $U\in \mathcal{U}(M)$ 
\[\lim\limits_{t\to\infty} \frac{1}{t}\lambda\left(\{s\in [0,t]\mid \varphi_s(x)\in U\}\right) =1 \] for all $x\in K$, where $\lambda$ denotes the Lebesgue measure on $[0,\infty)$,
\item[e)] \emph{stable in the sense of Lyapunov} if for all $U\in \mathcal{U}(M)$ there exists $V\in\mathcal{U}(M)$, $V\subseteq U$ such that 
\[\varphi_t(V)\subseteq U \textrm{ for all } t\geq 0 \, . \]
\end{itemize}
\end{defi}
The concepts {\normalfont{a), b)}} and {\normalfont{e)}} in \Cref{DS} have been established by A. M. Lyapunov in his dissertation (\cite{lya}) in 1892 and have since been broadly applied and investigated for dynamics on locally compact metric spaces. See \cite[Chapt. II]{bhatia} or \cite[Chapt. 3, Sect. 2]{denker}. The property {\normalfont{d)}} in \Cref{DS} appears in G. D. Birkhoff's monograph ``Dynamical Systems'' \cite[Chapt. VII]{birkhoff} as ``central motion'' and has been further investigated by H. Hilmy, see for example \cite{hilmy}, K. Sigmund, in \cite{sigmund} and by H. Kreidler in \cite[Sect. 4]{hekr} to name a few. Definition {\normalfont{c)}} for semiflows on smooth compact manifolds is due to J. Milnor and can be found in \cite[Section 2]{milnor}.
\begin{rem} If $(K,(\varphi_t)_{t\geq 0})$ is a dynamical system with metric $K$ then there exists one $\mu$-null set satisfying the assumptions in \Cref{DS} {\normalfont{c)}} that does not depend on $U\in\mathcal{U}(M)$ since there is a countable neighborhood basis and the countable union of null sets is again a null set. 
\end{rem}
\begin{rem} For the concepts defined in \Cref{DS} the following implications hold.
{\begin{center}\begin{xy}
  \xymatrix{a) \ar@{=>}[r] \ar@{=>}[rd]&b)\ar@{=>}[r]\ar@{=>}[rd]&c) & \textrm{and} & a) \ar@{<=>}[r] &b)+e) \\
&e) &d)&&}
\end{xy}\end{center}}
  
The opposite implications do not hold true in general which can be seen in the following examples. The equivalence of \Cref{DS} {\normalfont{a)}} $\iff$ {\normalfont{b)}} + {\normalfont{e)}} will be proven in below \Cref{K} and \Cref{kompakt}. 
\begin{exm}
\end{exm}
\begin{itemize}
\item[a)] Consider $K:=\mathbb{R}\cup\{\infty\}$ the one-point compactification of $\mathbb{R}$ and the semiflow $(\varphi_t)_{t\geq 0}$ defined by 
\[\varphi_t(x):= \begin{cases}x+t &x\in \mathbb{R} \\ \infty &x=\infty \end{cases} \, . \]
Then $M:=\{\infty\}$ is attractive but not uniformly attractive.
\item[b)] Take $K:=[0,\infty]$ the one-point compactification of $[0,\infty)$ and the semiflow $(\varphi_t)_{t\geq 0}$ on $K$, with 
\[\varphi_t(x):= \begin{cases}e^{-t}x &x\in [0,\infty) \\ \infty &x=\infty \end{cases} \, . \] Consider the standard Gaussian measure $\gamma$ on $[0,\infty]$ which is a regular Borel measure on $K$ that is quasi-invariant with respect to $(\varphi_t)_{t\geq 0}$ since it is equivalent to the Lebesgue measure $\lambda$. In particular,  $\gamma(\{\infty\})=0$. Then $M:=\{0\}$ is a likely limit set for $\gamma$ since $\gamma([0,\infty))=1$ and $\varphi_t(x)\to 0$ for all $x\in [0,\infty)$ but it is neither attractive nor a center of attraction since $\varphi_t(\infty)=\infty$ for all $t\geq 0$. 
\item[c)]In \cite[p.287]{hilmy}, H. Hilmy gave a concrete example for a center of attraction that is not attractive. We give a simplified version of this example. Take the following differential equation 
\[\left\{\begin{array}{ll} \dot{r}= -r\log(r)\left((1-r)^2+\sin^2(\theta)\right) \\
         \dot{\theta}= (1-r)^2+\sin^2(\theta)\end{array}\right. \, .\]
given in polar coordinates on $K:=\{z\in\mathbb{C}\mid 1\leq |z|\leq 2\}$. The solutions of above differential equation exist for all times and all initial values in $K$ and form a semiflow $(\varphi_t)_{t\geq 0}$ thereon. \\[1em]

The dependence of $r(t)$ and $\theta(t)$ is given by 
\[ r(t)=e^{C\cdot e^{-\theta(t)}} \quad \leftrightarrow \quad \theta(t)=-\log(\log(r(t)))+\log(C) \, . \]

The orbit of an initial state with radius $r>1$ forms a spiral towards the unit circle. On the unit circle the radius is constant and the rate of change of $\theta(t)$ is given by the differential equation
\[\dot{\theta}=\sin^2(\theta) \, . \]
Thus, $z_1=1=e^0$ und $z_2=-1=e^{\pi i}$ are fixed points, because $\sin^2(0)=\sin^2(\pi)=0$. Therefore, the orbits of states on the unit circle converge to either $z_1$ or $z_2$.
The set $M:=\{z_1\}\cup\{z_2\}$ is a center of attraction for $(K,(\varphi_t)_{t\geq 0})$ and it is even minimal with this property. It is easy to see that the minimal attractive subset in this example is $\mathbb{T}:=\{z\in \mathbb{C}\mid |z|=1\}$.
\end{itemize}
\end{rem}
The next proposition characterizes all above mentioned attractivity properties by means of the corresponding Koopman semigroup. 
\begin{prop}\label{main}Let $(K,(\varphi_t)_{t\geq 0})$ be a dynamical system, $\emptyset\neq M \subseteq K$ a closed invariant set and $(S(t))_{t\geq 0}$ the restricted Koopman semigroup, i.e. $S(t):=T(t)|_{I_M}$ for $t\geq 0$. \begin{itemize}
\item[(I)] The following are equivalent.
\begin{itemize}
\item[a)] $(S(t))_{t\geq 0}$ is strongly stable.
\item[b)] $M$ is uniformly attractive.
\end{itemize}
\item[(II)] The following are equivalent
\begin{itemize}
\item[a)] $(S(t))_{t\geq 0}$ is weakly stable.
\item[b)] $M$ is attractive.
\end{itemize}
\item[(III)] The following are equivalent.
\begin{itemize}
\item[a)] $(S(t))_{t\geq 0}$ is almost everywhere pointwise stable.
\item[b)] $M$ is a likely limit set. 
\end{itemize}
\item[(IV)] The following are equivalent.
\begin{itemize}
\item[a)] $(S(t))_{t\geq 0}$ is almost weakly stable.
\item[b)] $M$ is a center of attraction. 
\end{itemize}
\end{itemize}
\end{prop}
\begin{proof} Proof of {\normalfont{(I)}}: First we show {\normalfont{a)$\implies$b)}}. Take $U\in \mathcal{U}(M)$. Since $K\setminus M$ is completely regular there is $f\in I_M$ and $\eps>0$ such that $U_{\eps,f}\subseteq U$. By assertion {\normalfont{a)}} there is $t_0>0$ such that $\|S(t)f\|<\eps$ for all $t\geq t_0$. This implies \[|S(t)f(x)|=|f(\varphi_t(x))|<\eps\textrm{ for all }x\in K\setminus M\, , \ t\geq t_0\, .\] Therefore, $\varphi_t(K)\subseteq U_{\eps, f}\subseteq U$ for all $t\geq t_0$. Also {\normalfont{b)$\implies$a)}} because for every $\eps>0$ and $f\in I_M$ there is a $t_0>0$ such that $\varphi_t(K)\subseteq U_{\eps,f}$ for all $t\geq t_0$. This implies $|S(t)f(x)|<\eps$ for all $t\geq t_0$ and $x\in K$ and therefore
$\|S(t)f\|=\sup\limits_{x\in K}|S(t)f(x)|<\eps$ for all $t\geq t_0$.\\[1em]

Proof of {\normalfont{(II)}}: To prove {\normalfont{a)$\implies$b)}} take $U\in\mathcal{U}(M)$ and $x\in K$. Then there exist $\eps>0$ and $f\in I_M$ such that $U_{\eps,f}\subseteq U$ and since $(S(t))_{t\geq 0}$ is weakly stable there exists $t_0>0$ such that 
\[\langle S(t)f,\delta_x\rangle =f(\varphi_t(x)) < \eps \textrm{ for all }t\geq t_0 \] which implies $\varphi_t(x)\in U_{\eps,f}\subseteq U$ for all $t\geq t_0$. For the opposite implication let $\eps>0$, $f\in I_M$ and $x \in K$ by {\normalfont{b)}} there exists $t_0>0$ such that $\langle S(t)f,\delta_x\rangle < \eps $ for all $t\geq t_0$ and thus
\[ \langle S(t)f, \delta_x \rangle \to 0 \textrm{ as } t\to \infty \] for all Dirac measures $\delta_x$ and by Lebesgue's theorem of dominated convergence
\[ \langle S(t)f, \mu \rangle \to 0 \textrm{ as } t\to \infty \] for all $\mu\in I_M'$.\\[1em]

Proof of {\normalfont{III}}: We first prove {\normalfont{a)}} implies {\normalfont{b)}}. First take a neighborhood $U\in \mathcal{U}(M)$, then there exist $f\in I_M$ and $\varepsilon> 0$ with $U_{\varepsilon,f}\subseteq U$. By assumption there exists a quasi invariant Borel measure $\mu$ and $\mu$-null set $N_f$ depending on $f$ such that for every $x\in N_f^c$ there is $t_0>0$ such that
\[T(t)f(x)< \varepsilon\] for all $t\geq t_0$.  Clearly, this implies $\varphi_t(x)\in U_{\varepsilon,f}\subseteq U$ for all $t\geq t_0$. The other implication follows similarly. \\[1em]
Proof of {\normalfont{(IV)}}: First recall that $I_M\cong \C_0(K\setminus M)$ by $f\mapsto f|_{K\setminus M}$. Let $M$ be a center of attraction and $U\in \mathcal{U}(M)$ open, i.e.,
\[\lim\limits_{t\to \infty}\frac{1}{t}\, \lambda(\{s\in[0,t]\mid \phy_s(x)\in U \})=1 \] for all $x\in K$ and consequently for $A:=K\setminus U$, which is a compact subset of $K\setminus M$, the following equivalence holds.
\begin{align*}&\lim\limits_{t\to \infty}\frac{1}{t}\, \lambda(\{s\in[0,t]\mid \phy_s(x)\in A \})=0 \\
\iff &\lim\limits_{t\to \infty}\frac{1}{t}\, \int\limits_0^t \! \mathds{1}_A(\varphi_s(x)) \, \mathrm{d}s =0\, .
\end{align*}
Now take $f\in I_M$ with compact support and denote $\mathrm{supp}(f)=:A$. Then $|f|\leq \|f\|\cdot\mathds{1}_A=\mathds{1}_A$ and thus
\[\lim\limits_{t\to \infty}\frac{1}{t}\, \int\limits_0^t \! |T(s)f|(x) \, \mathrm{d}s =0\, . \] Note that the mapping
\[(s,x)\mapsto |T(s)f|(x) \] is continuous and hence measurable, thus for $t>0$ fixed 
\begin{align*}\frac{1}{t}\, \int\limits_0^t \! |T(s)f|(x) \, \mathrm{d}s &= \frac{1}{t}\, \int\limits_0^t \! \int\limits_K \! |T(s)f|\, \mathrm{d}\delta_x \, \mathrm{d}s \\
&=  \int\limits_K \! \frac{1}{t}\, \int\limits_0^t \! |T(s)f|\, \mathrm{d}s \,\mathrm{d}\delta_x \xrightarrow{t\to\infty} 0
\end{align*} by the Fubini-Tonelli Theorem. 
And thus by Lebesgue's Theorem of dominated convergence
\begin{align*}\int\limits_K \! \frac{1}{t}\, \int\limits_0^t \! |T(s)f|\, \mathrm{d}s \,\mathrm{d}\mu \geq
\frac{1}{t}\, \int\limits_0^t \! |\langle T(s)f,\mu\rangle | \, \mathrm{d}s \xrightarrow{t\to\infty}0\, . \end{align*} for all $\mu\in I_M'$.
This implies $\C_c(K\setminus M)\subseteq I_{\textrm{aws}}$ and since $I_{\textrm{aws}}$ is closed and the continuous functions with compact support are dense in $\C_0(K\setminus M)$ it follows that
\[\C_0(K\setminus M)\cong I_M\subseteq I_{\textrm{aws}}\, .\]

On the other hand since $I_{\textrm{aws}}$ is closed invariant ideal there exists $L\subset K$ closed and invariant such that $I_{\textrm{aws}}=I_L$. Take $A\subseteq K\setminus L$ compact and $\mathds{1}\in\C(A)$ then there exists a continuous extension $g\in \C_0(K\setminus L)\cong I_L$ and the measurable extension $\mathds{1}_A$ as the characteristic function on $A$. Since $\mathds{1}_A\leq g$ the following equation is true for all $x\in K$
\[\lim\limits_{t\to \infty}\frac{1}{t}\, \int\limits_0^t \! |T(s)g|(x) \, \mathrm{d}s =0 \] hence
\[\lim\limits_{t\to \infty}\frac{1}{t}\, \int\limits_0^t \! T(s)\mathds{1}_A(x) \, \mathrm{d}s =0\, , \]
which implies
\[\lim\limits_{t\to \infty}\frac{1}{t}\,\lambda(\{s\in [0,t]\mid \varphi_s(x)\in A\})=0 \] and thus
\[\lim\limits_{t\to \infty}\frac{1}{t}\,\lambda(\{s\in [0,t]\mid \varphi_s(x)\in K\setminus A\})=1 \, . \]
The open neighborhoods of $M$ are exactly the complements of compact sets $A$ and since $A$ was arbitrary the last equation holds for all open neighborhoods $U$ of $M$. Thus $M$ is a center of attraction.
\end{proof}
To conclude this section we show that the concepts of uniform attractivity and pointwise attractivity coincide if and only if $M$ is stable in the sense of Lyapunov. To do so we first characterize stability in the sense of Lyapunov further.

\begin{prop}\label{stable}Let $\emptyset\neq M \subseteq K$ be closed and invariant. Then the following are equivalent.
\begin{itemize}
\item[a)]The set $M$ is stable in the sense of Lyapunov.
\item[b)] Every $U\in\mathcal{U}(M)$  contains an invariant $V\in \mathcal{U}(M)$. If $U$ is closed, $V$ can be chosen closed as well.
\end{itemize}
\end{prop}
\begin{proof}For the implication {\normalfont{a)$\Rightarrow$b)}} take $U\in\mathcal{U}(M)$ closed and $V\in \mathcal{U}(M)$, $V\subseteq U$ such that $x\in V$ implies $\varphi_t(x)\in U$ for all $t\geq 0$. Consider $W:=\overline{\bigcup_{t\geq 0}\varphi_t(V)}$. Then $V\subset W\subset U$. Therefore, $W$ is still a closed neighborhood of $M$ which is invariant. The implication {\normalfont{b)$\Rightarrow$a)}} is trivial.\\
\begin{rem} Furthermore, a closed invariant set $\emptyset\neq M\subseteq K$ is Lyapunov stable if and only if \[M=\bigcap\limits_{V\in\mathcal{U}(M)\textrm{ inv.}} V \, . \]
\end{rem}
We prove that the assertion implies {\normalfont{b)}} in \Cref{stable}. Let $U$ be an open neighborhood of $M$ and assume there is no invariant neighborhood $V$ of $M$ with $V\subseteq U$. Then for all invariant neighborhoods $V$ there exists $x_V\in V$ with $x_V\in U^c$. This defines a net $(x_V)_{V\textrm{ inv.}}$ which has a convergent subnet, since $U^c$ is compact. Let $\overline{x}$ be the limit of said convergent subnet. On the other hand there exists a convergent subnet with limit in $M$ since $K$ is compact and the sets $V$ are neighborhoods of $M$. This implies $\overline{x}\in M=\bigcap\limits_{V\in\mathcal{U}(M)\textrm{ inv.}} V$ which contradicts the fact that $\overline{x}\in U^c$. This implies {\normalfont{b)}}.
\end{proof}
\begin{prop}\label{K}Let $(K,(\varphi_t)_{t\geq 0})$ be a dynamical system and $\emptyset\neq M\subseteq K$ a closed invariant subset which is uniformly attractive, then $M$ is stable in the sense of Lyapunov. 
\end{prop}
\begin{proof} Let $M$ be uniformly attractive and assume $M$ is not stable in the sense of Lyapunov then there exists $f\in I_M$, $f\geq 0$ and $\varepsilon>0$ such that $U_{\varepsilon,f}$ contains no invariant neighborhood. This implies there exists a net $(x_i)_{i\in I}\subseteq U_{\varepsilon,f}$ and $(t_i)_{i\in I}\subseteq [0,\infty)$ such that 
\[T(t_i)f(x_i)\geq\varepsilon\textrm{ for all } i \in I \,. \] Thus, $\left(T(t_i)f(x_i)\right)_{i\in I}$ has a convergent subnet, converging to $y\in[f\geq \varepsilon]$. \\

Since there exists $t_0>0$ such that $\|T(t)f\|< \varepsilon$ for all $t\geq t_0$ since $M$ is uniformly attractive by assumption, the net $(t_i)_{i\in I}$ is bounded and thus there exists a convergent subnet $(t_{i_j})_{j\in J}$ with $t_{i_j}\to t^* \leq t_0$. Also, since $U_{\varepsilon,f}$ is compact there exists a convergent subnet $(x_{i_j})_{j\in J}$ with $x_{i_j}\to x\in M$. By continuity $T(t_{i_j})f(x_{i_j})\to T(t^*)f(x) \in M$, since $M$ is invariant. By uniqueness of the limit $M\ni T(t^*)f(x) =y \in [f\geq\varepsilon]$ which is a contradiction. Therefore, $M$ must be stable in the sense of Lyapunov.
\end{proof}

On the other hand stability in the sense of Lyapunov and attractivity imply uniform attractivity. 
\begin{prop}\label{kompakt}
Let $(K,(\varphi_t)_{t\geq 0})$ be a dynamical system and $\emptyset \neq M \subseteq K$ attractive and stable in the sense of Lyapunov. Then $M$ is uniformly attractive.
\end{prop}
\begin{proof}
Let $M$ be pointwise attractive and stable in the sense of Lyapunov. Take $U\in \mathcal{U}(M)$ open and invariant. Then for every $x\in K$ there exists $t_0=t_0(x,U)$ such that 
\[\varphi_t(x)\in U \textrm{ for all } t\geq t_0 \, . \] 
Since $\varphi_{t_0}$ is continuous there exists an open neighborhood $U_x$ of $x$ with 
\[\varphi_{t}(U_x)\subseteq U \textrm{ for all }t\geq t_0 \]
since $U$ is invariant. Now since $K$ is compact there exist $x_1,\dots,x_n\in K$ for some $n\in \mathbb{N}$ such that $K\subseteq \bigcup\limits_{i=1}^n U_{x_i}$ and for $t\geq \max\limits_{i=1,\dots,n} t_0(x_i,U)$
\[\varphi_t(K)\subseteq \varphi_t\left(\bigcup\limits_{i=1}^n U_{x_i}\right)\subseteq \bigcup\limits_{i=1}^n \varphi_t(U_{x_i})\subseteq U \, . \] This implies the assertion. 
\end{proof}

%%%%%%%%%%%%%%%%%%%%%%%%%%%%%%%
\section{Existence and Characterization of Minimal Attractors}
Given a dynamical system $(K,(\varphi_t)_{t\geq 0})$ there always exist attractors in the sense of above \Cref{DS} {\normalfont{a),b),c)}} and {\normalfont{d)}} since the subspaces $I_{\textrm{ss}}$, $I_{\textrm{ws}}$, $I_{\textrm{aws}}$ and $I_{\textrm{aews}}$ are all closed ideals of $\mathrm{C}(K)$ and are maximal with this property. We thus obtain a corresponding closed invariant set $\emptyset\neq M\subseteq K$ that is uniformly attractive, attractive, a likely limit set or a center of attraction and it is minimal with this property by construction. In this subsection we will discuss what the corresponding minimal attractor $M$ looks like. 
First, we clarify that $I_{\textrm{ss}}$, $I_{\textrm{ws}}$, $I_{\textrm{aws}}$ and $I_{\textrm{aews}}$ are in fact not only closed subspaces but ideals in $\mathrm{C}(K)$. 
\begin{rem} The closed subspaces $I_{\textrm{ss}}$, $I_{\textrm{ws}}$ and $I_{\textrm{aws}}\subseteq \mathrm{C}(K)$ are order or equivalently algebra ideals in $\mathrm{C}(K)$.
\end{rem}
\begin{proof}We only compute this for $I_{\textrm{aws}}$, because $I_{\textrm{ws}}$ and $I_{\textrm{aews}}$ follow analoguosly and $I_{\textrm{ss}}$ is clearly an order ideal. \\
By \Cref{awsi}, $I_{\textrm{aws}}$ is a closed subspace of $\mathrm{C}(K)$. It remains to show that it is an algebra or equivalently a lattice ideal. 
Now take $f\in I_{\textrm{aws}}$ we have to show that $|f|\in I_{\textrm{aws}}$. Take $x\in K$ and recall that for every $t\geq 0$, $|\langle T(t)f,\delta_x\rangle|=|T(t)f(x)|=T(t)|f|(x)$. This implies
\begin{align*}0 \xleftarrow[]{T\to \infty}&\frac{1}{T}\int\limits_0^T \! |\langle T(t)f, \delta_x \rangle| \, \mathrm{d}t \\
=&\frac{1}{T}\int\limits_0^T \! \langle T(t)|f|, \delta_x \rangle \, \mathrm{d}t \\
=& \langle \frac{1}{T}\int\limits_0^T \! T(t)|f| \, \mathrm{d}t, \delta_x \rangle \, .
\end{align*} Applying Lebesgue's Theorem of dominated convergence implies the assertion. 
\end{proof}
\subsection{Uniform attractivity}
The following proposition gives a characterization of uniform attractivity. 
\begin{prop}\label{unif}Let $(K,(\varphi_t)_{t\geq 0})$ be a dynamical system and $\emptyset\neq M\subseteq K$ closed and invariant. Then the following are equivalent.
\begin{itemize}
\item[a)] The set $M$ is uniformly attractive,
\item[b)] $\bigcap\limits_{t\geq 0}\varphi_t(K)\subseteq M$.
\end{itemize}
\end{prop}
\begin{proof} 
To prove {\normalfont{a)$\implies$b)}} note that 
\[\bigcap\limits_{t\geq 0}\varphi_t(K)\subseteq \bigcap\limits_{t\in T}\varphi_t(K)\subseteq \bigcap\limits_{U\in \mathcal{U}(M)}U =M \]
where 
\[T:=\{t_0\geq 0 \mid \varphi_t(K)\subseteq U\textrm{ for some }U\in \mathcal{U}(M)\, , \ t\geq t_0 \}\, . \]

The opposite implication {\normalfont{b)$\implies$a)}} follows since for every $U\in \mathcal{U}(M)$ the following chain holds
\[ \bigcap\limits_{t\geq 0}\varphi_t(K)\subseteq M =\bigcap\limits_{V\in\mathcal{U}(M)}V\subseteq U \, . \]
\end{proof}
As an immediate result we obtain the following.
\begin{prop}Let $(K,(\varphi_t)_{t\geq 0})$ be a dynamical system. Then there exists a unique minimal uniformly attractive subset of $K$ given by 
\[\bigcap_{t\geq 0}\varphi_t(K)\, . \]
\end{prop}
\begin{proof} The set $\bigcap_{t\geq 0}\varphi_t(K)$ is closed and $(\varphi_t)_{t\geq 0}$-invariant and is uniformly attractive by \Cref{unif} {\normalfont{b)}} and is minimal with this property by construction.
\end{proof}
\begin{rem}\[I_{\textrm{ss}}=I_{\bigcap\limits_{t\geq 0}\varphi_t(K)} \, . \]
\end{rem}
%%%%%%%%%%%%%%%%%%%%%%
\subsection{Attractivity and $\omega$-limit sets}
It is useful to introduce $\omega$-limit sets to study asymptotic properties of dynamical systems. Similar concepts have already been used by H. Poincaré, but G. D. Birkhoff  first introduced the term $\omega$-limit set in \cite[Chapt. VII, p.198]{birkhoff}. The characterization of attractors via $\omega$-limit sets is due to N.P. Bhatia and G.P. Szegö and can be found in \cite[Chapt. II]{bhatia}.
\begin{defi}
For $x\in K$ we define the $\omega$-limit set of $x$ as 
\[\omega(x):=\bigcap\limits_{T \geq 0} \overline{\{\varphi_t(x)\mid t\geq T\}} \, . \]
\end{defi}
\begin{prop} For every $x\in \Omega$
\[ \omega(x)=\{y\in \Omega \mid \exists \textrm{ a net }(t_i)_{i\in I}\textrm{ in }[0,\infty)\, ,\ t_i\to \infty\textrm{ such that }\varphi_{t_i}(x)\to y\, , i\in I\} \, . \]
\end{prop}
\begin{proof} Take $x\in \Omega$ and $y\in \omega(x)$, i.e., $y\in \overline{\{\varphi_t(x)\mid t\geq T\}}$ for all $T\geq 0$. In particular, $y\in \overline{\orb(x)}$. Thus, by definition of the closure there exists a net in $\orb(x)$ converging to $y$. For the other implication let $(t_i)_{i\in I}$ be a net with $t_i\to \infty$ such that $\varphi_{t_i}(x)$ converges to $y$. For fixed $T\geq 0$ there exists $i_0\in I$ such that $t_i \geq T$ for all $i\geq i_0$. Then $(\varphi_{t_i}(x))_{i\in I, i\geq i_0}$ is still a net converging to $y$. Therefore, $y\in \overline{\{\varphi_t(x)\mid t\geq T\}}$ for all $T\geq 0$. Thus follows the assertion. 
\end{proof}Next, we discuss some properties of $\omega$-limit sets.
\begin{prop} The set $\omega(x)$ is non-empty, closed and invariant under $(\varphi_t)_{t\geq 0}$ for all $x\in \Omega$.
\end{prop}
\begin{proof} Take $x\in \Omega$. The set $\omega(x)$ is closed by definition as an intersection of closed sets and non-empty by the finite intersection property of $K$. For the invariance take $r>0$ and $y\in \omega(x)$. Then there exists a net $(\varphi_{t_i}(x))_{i\in I}$ converging to $y$ for $i\in I$, since $\varphi_r$ is continuous, $(\varphi_r(\varphi_{t_i}(x))_{i\in I}$ converges to $\varphi_r(y)$, thus $\varphi_r(y)\in\omega(x)$.  \end{proof}
\begin{prop}\label{omega} Let $(K,(\varphi_t)_{t\geq 0})$ be a dynamical system and $\emptyset\neq M \subseteq K$ closed and invariant. Then the following are equivalent.
\begin{itemize}
\item[a)] The set $M$ is attractive,
\item[b)] $\omega(x)\subseteq M$ for all $x\in K$.
\end{itemize}
\end{prop}
\begin{proof} To prove {\normalfont{a)$\implies$b)}} take $x\in K$. By {\normalfont{a)}} 
\[\omega(x)\subseteq \bigcap_{U\in\mathcal{U}(M)}U=M \, . \]
Consider $U\in\mathcal{U}(M)$ open and assume that {\normalfont{a)}} does not hold, i.e., there exists $x\in K\setminus M$ with $\varphi_t(x)\in U^c$ for infinitely many $t>0$. Since $U^c$ is closed and hence compact there exists a convergent subnet $(t_i)_{i\in I}$, $t_i\to \infty$ such that $\varphi_{t_i}(x)\to z\in U^c$ which is a contradiction to {\normalfont{b)}}.
\end{proof}
\begin{prop} Let $(K,(\varphi_t)_{t\geq 0})$ be a dynamical system. Then there exists a unique minimal attractive subset of $K$ given by
\[\overline{\bigcup_{x\in K}\omega(x)}\, .  \]
\end{prop}
\begin{proof}
In \Cref{omega} {\normalfont{b)}} we have seen that $\omega(x)$ is contained in every closed, $(\varphi_t)_{t\geq 0}$-invariant and attractive subset $\emptyset\neq M\subseteq K$ therefore also
\[\bigcup_{x\in K}\omega(x)\subseteq M\, . \] Also the closure $\overline{\bigcup_{x\in K}\omega(x)}$ is contained in every such $M$ and $(\varphi_t)_{t\geq 0}$-invariant, attractive itself and minimal with this property by construction. 
\end{proof}
\begin{rem}\[I_{\textrm{ws}}=I_{\overline{\bigcup\limits_{x\in K}\omega(x)}} \, . \]
\end{rem}
\begin{prop}\label{fast} Let $(K,(\varphi_t)_{t\geq 0})$ be a dynamical system with $K$ metric, $\mu$ a quasi invariant regular Borel measure on $K$ and $\emptyset\neq M \subseteq K$ closed and invariant. Then the following are equivalent.
\begin{itemize}
\item[a)] The set $M$ is a likely limit set,
\item[b)]  there exists a quasi invariant regular Borel measure on $K$ such that $\omega(x)\subseteq M$ for $\mu$-almost every $x\in K$.
\end{itemize}
\end{prop}
\begin{proof} This follows directly from \Cref{omega}.
\end{proof}
\begin{rem} Let $(K,(\varphi_t)_{t\geq 0})$ be a dynamical system with $K$ metric, $\mu$ a quasi invariant regular Borel measure on $K$, then there exists a $\mu$-null set $N$ such that
\[I_{\textrm{aews}}=I_{\overline{\bigcup\limits_{x\in N^c}\omega(x)}} \, . \]
\end{rem}

%\begin{prop} Let $(\Omega,(\varphi_t)_{t\geq 0})$ be a dynamical system with locally-compact state space $\Omega$ then there exists a compact subset $M\subset \Omega$ which is a likely limit set and minimal with this property. 
%\end{prop}
%\begin{proof}
%\end{proof}
%%%%%%%%%%%%%%%%%%%%%%%%%%%%%%%%%%%%%%%
\subsection{Minimal centers of attraction and ergodic measures}
An interesting fact is that the minimal center of attraction is characterized by the ergodic probability measures on $K$. We recall that a regular Borel measure is called \emph{invariant} if $\mu(\varphi_t^{-1}(A))=\mu(A)$ for all Borel measurable sets $A$ and $t\geq 0$. We then call an invariant regular Borel measure ergodic if the corresponding measure-preserving system $(K,(\varphi_t)_{t\geq 0}, \mu)$ is ergodic, i.e. if $A\subset K$ Borel measurable and invariant then $\mu(A)\in \{0,1\}$.
\begin{prop}\label{mcoa}
The minimal center of attraction is given by the union of supports of ergodic measures, i.e.,
\[I_{\textrm{aws}}=I_{\overline{\bigcup\limits_{\mu\textrm{ erg. }}\textrm{supp}(\mu)}} \, . \]
\end{prop}
\begin{proof}
We have to show that \[I_{\textrm{aws}}=I_{\overline{\bigcup\limits_{\mu\textrm{ inv.}}\textrm{supp}(\mu)}}\, .\]
First we show ``$\subseteq$''. Let $\mu\in \C(K)'$ be an invariant measure. For $f\in I_{\textrm{aws}}$
\begin{align*}\langle |f|,\mu\rangle = &\frac{1}{t}\int_0^t\! \langle |f|,\mu\rangle \, \mathrm{d} s\\
\overset{\mu \textrm{ inv.}}{=} &\frac{1}{t}\int_0^t\! \langle T(s)|f| , \mu \rangle\, \mathrm{d}s\\
\overset{T(t)\textrm{ lattice hom.}}{=}  &\frac{1}{t}\int_0^t\! \langle |T(s)f| , \mu \rangle\, \mathrm{d}s\to 0 \, . \end{align*}
Therefore, $f\big|_{\mathrm{supp}\mu}\equiv 0$ for all invariant measures $\mu$. For the implication ``$\supseteq$'' let $x\in K$ and $\delta_x$ the corresponding Dirac measure and $f\in I_{\textrm{aws}}$. We observe that 
\begin{align*}\frac{1}{t}\int_0^t\! |\langle f,T(s)'\delta_x\rangle|\, \mathrm{d}s &= \frac{1}{t}\int_0^t\! \langle |f|,T(s)'\delta_x\rangle\, \mathrm{d}s\\
 &= \langle|f|,\frac{1}{t}\int_0^t\! T(s)'\delta_x\, \mathrm{d}s\rangle \, .
 \end{align*}
Since the dual unit ball $B'$ is compact in the weak-*-topology and $\frac{1}{T}\int_0^T\! T(t)'\delta_x$ is bounded, every subnet of $\frac{1}{T}\int_0^T\! T(t)'\delta_x$ has a convergent subnet in $B'$, i.e.,
\[\langle|f|,\frac{1}{t_i}\int_0^{t_i}\! T(s)'\delta_x\rangle\, \mathrm{d}s \to \langle|f|, \mu\rangle \] with $\mu$ invariant. By \cite[Prop.3.5]{hekr} it follows that $ \overline{\bigcup\limits_{\mu\textrm{ inv. }}\textrm{supp}(\mu)}=\overline{\bigcup\limits_{\mu\textrm{ erg. }}\textrm{supp}(\mu)}$.
\end{proof}
\section*{Compliance with ethical standards}
Conflict of interest: The author declares that she has no conflict of interest. \\
Ethical approval: This article does not contain any studies with human participants or animals performed by the author. 

\printbibliography
\end{document}